\pdfoutput=1
\documentclass[letterpaper]{amsart}

\usepackage[utf8]{inputenc}
\usepackage{lmodern}
\usepackage{mathtools,amsfonts,amsthm,mathrsfs,amssymb,bm}
\mathtoolsset{showonlyrefs}

\usepackage[margin=1in]{geometry}

\usepackage[T1]{fontenc}

\usepackage[dvipsnames]{xcolor}
\usepackage{textcomp}

\usepackage[shortlabels]{enumitem}

\PassOptionsToPackage{pdfusetitle}{hyperref}
\usepackage{bookmark}
\usepackage{hypcap}
\usepackage{microtype}

\newtheorem{theorem}{Theorem}
\newtheorem{lemma}[theorem]{Lemma}
\newtheorem{proposition}[theorem]{Proposition}

\newtheorem{claim}[theorem]{Claim}

\theoremstyle{definition}
\newtheorem{definition}[theorem]{Definition}

\usepackage{cleveref}
\crefname{claim}{claim}{Claim}
\AddToHook{env/proposition/begin}{\crefalias{theorem}{proposition}}
\AddToHook{env/lemma/begin}{\crefalias{theorem}{lemma}}
\AddToHook{env/conjecture/begin}{\crefalias{theorem}{conjecture}}
\AddToHook{env/claim/begin}{\crefalias{theorem}{claim}}

\makeatletter
\newenvironment{claimproof}[1][Proof]{\par
    \pushQED{\qed}%
    
    \normalfont \topsep6\p@\@plus6\p@\relax
    \trivlist
    \item[\hskip\labelsep
    \textit{#1}\@addpunct{.}~]\ignorespaces
}{%
    \popQED\endtrivlist\@endpefalse%
}
\makeatother

\makeatletter
\newcommand{\step}[1]{%
  \par%
  \addvspace{\medskipamount}%
  \textit{#1\@addpunct{.}}\enspace\ignorespaces%
}
\makeatother

\NewDocumentCommand{\EE}{}{\mathbb{E}}
\NewDocumentCommand{\PP}{}{\mathbb{P}}
\NewDocumentCommand{\QQ}{}{\mathbb{Q}}
\NewDocumentCommand{\RR}{}{\mathbb{R}}

\NewDocumentCommand{\bh}{}{\mathbf{h}}
\NewDocumentCommand{\bq}{}{\mathbf{q}}
\NewDocumentCommand{\bu}{}{\mathbf{u}}
\NewDocumentCommand{\bv}{}{\mathbf{v}}
\NewDocumentCommand{\bw}{}{\mathbf{w}}
\NewDocumentCommand{\bx}{}{\mathbf{x}}
\NewDocumentCommand{\by}{}{\mathbf{y}}

\NewDocumentCommand{\bbeta}{}{\bm{\beta}}
\NewDocumentCommand{\bgam}{}{\bm{\gam}}
\NewDocumentCommand{\blam}{}{\bm{\lam}}
\NewDocumentCommand{\bmu}{}{\bm{\mu}}
\NewDocumentCommand{\btau}{}{\bm{\tau}}

\NewDocumentCommand{\bone}{}{\bm{1}}
\NewDocumentCommand{\bzero}{}{\bm{0}}

\NewDocumentCommand{\cE}{}{\mathcal{E}}
\NewDocumentCommand{\cG}{}{\mathcal{G}}
\NewDocumentCommand{\cI}{}{\mathcal{I}}

\NewDocumentCommand{\gam}{}{\gamma}
\NewDocumentCommand{\Gam}{}{\Gamma}
\NewDocumentCommand{\lam}{}{\lambda}

\DeclareMathOperator{\diag}{diag}
\DeclareMathOperator{\mad}{mad}

\DeclarePairedDelimiter{\abs}{\lvert}{\rvert}
\DeclarePairedDelimiter{\norm}{\lVert}{\rVert}
\DeclarePairedDelimiterXPP{\opnorm}[1]{}{\lVert}{\rVert}{_{\mathrm{op}}}{#1}

\RenewDocumentCommand{\epsilon}{}{\varepsilon}
\RenewDocumentCommand{\subset}{}{\subseteq}
\RenewDocumentCommand{\Pr}{}{\PP}
\NewDocumentCommand{\trans}{}{^{\mathsf{T}}}

\DeclareMathOperator{\ind}{Ind}

\title{Degree-sequence bounds for independent sets via multivariate local occupancy}

\date{\today}

\author[E.\ Davies]{Ewan Davies}
\address{Department of Computer Science, Colorado State University, Fort Collins CO, USA}
\email{research@ewandavies.org}
\thanks{E.D.\ supported in part by NSF grant CCF-2309707.}

\author[JS.\ Sandhu]{Juspreet Singh Sandhu}
\address{Department of Computer Science, Colorado State University, Fort Collins CO, USA}
\email{js.sandhu@colostate.edu}

\author[J.\ Seo]{Jaehyeon Seo}
\address{Department of Mathematics, Yonsei University, Seoul, South Korea}
\email{jaehyeonseo@yonsei.ac.kr}
\thanks{J.S.\ supported by Samsung STF Grant SSTF-BA2201-02 and the National Research Foundation of Korea (NRF) grant MSIT NRF-2022R1C1C1010300. This research was done while J.S.\ was visiting Caltech, hosted by David Conlon and supported by the BK21 FOUR Support Program for Outstanding Graduate Students' International Joint Training.}

\author[B.\ Tan]{Brian Tan}
\address{Department of Computer Science, Colorado State University, Fort Collins CO, USA}
\email{Brian.Tan@colostate.edu}

\begin{document}
\begin{abstract}
We present new degree-sequence lower bounds on the expected size of an independent set from the hard-core model. 
For arbitrary graphs, we establish a multivariate lower bound inspired by a conjecture of the first author and Kang and a recent bound on the multivariate partition function due to Lee and the third author.
By applying a novel spectral analysis to the local occupancy linear program, our method successfully bypasses the convergence radius limitations of the cluster expansion and avoids induction. 
For graphs with bounded local maximum average degree, including triangle-free graphs, we prove a univariate bound extending prior work by a subset of the authors. 
In both cases our bounds require the fugacities to be upper bounded by $c/\Delta$ where $\Delta$ is the maximum degree of the graph and $c$ is an absolute constant depending on the setting.
\end{abstract}

\maketitle

\section{Introduction}

Problems related to understanding the size and number of independent sets in graphs have a long history and find applications in a variety of areas of mathematics and computer science. 
We are interested in results that give lower bounds on the independence number, the size $\alpha(G)$ of a largest independent set in a graph $G$, subject to local information such as degree and the sparsity of neighborhoods. 
Generally, we seek bounds in terms of local structure in the graph such as the maximum degree, average degree, or degree sequence and we consider local density conditions such as graphs having few triangles. 
A small family of proof techniques exist for such bounds: induction can be used to prove the Caro--Wei theorem~\cite{Car79,Wei81} (and generalizations~\cite{KP24}), and induction together with some convexity is the main idea in the proofs of Shearer's well-known results for triangle-free graphs~\cite{She83,She91}, and recent developments thereof~\cite{MS25}. 
More probabilistic methods include analyzing a uniform random independent set~\cite{Alo96,She95a}, and a generalization of this involves the so-called \emph{hard-core model} that forms the basis of this work~\cite{DJPR17b}. There are also other distributions on independent sets that give rise to results of the type we seek~\cite{AS16,BFS+26,MS25}.

The hard-core model generalizes the uniform distribution on independent sets in such a way that a useful spatial independence property is preserved. 
Consider a graph $G=(V,E)$ and the set $\cI(G)$ of independent sets in $G$. For a \emph{fugacity} vector $\blam\in[0,\infty)^V$ we define for any $I\in\cI(G)$ the measure
\[ \Pr_{G,\blam}(I) = \frac{1}{Z_G(\blam)}\prod_{v\in I}\lam_v, \] 
where $Z_G(\blam) = \sum_{J\in \cI(G)}\prod_{v\in J}\lam_v$ is the normalizing constant known as the \emph{partition function} that makes this a probability measure.
It is not too hard to show that distributions on $\cI(G)$ with the spatial Markov property correspond to those described as hard-core models, and that setting $\blam=\bone=(1,\dotsc,1)$ yields the uniform distribution on $\cI(G)$.
We let $\EE_{G,\blam}$ denote expectation with respect to the above hard-core measure and let $X\sim \Pr_{G,\blam}$ denote the random independent set. It is convenient to identify $X\in \cI(G)$ with the indicator vector $X\in\{0,1\}^V$ such that $X_v = 1$ iff $v\in X$.

Perhaps because of the utility of spatial independence, or the fact that such distributions are entropy-maximizing subject to given marginals $\Pr(X_v)$ for $v\in V$, the hard-core model is a useful tool in understanding the size and structure of independent sets in a surprising number of contexts. 
A substantial body of work shows that a sufficiently strong \emph{local} understanding of marginals in the hard-core model can be one of the best-known tools to understand lower bounds on independence number~\cite{DJPR17b,DK25} and upper bounds on (fractional and list) chromatic number~\cite{DJKP20,DdKP21,MR02} of locally sparse graphs. 
Further applications include determining the location of computational thresholds~\cite{DP23,DL24} and establishing algorithms for finding large independent sets~\cite{LMR+24} or colorings~\cite{DKPS20a}.

We are interested in pushing these methods further in the context of extremal bounds motivated by Ramsey theory and classic extremal graph theory. 
The independence number $\alpha(G)$ is the size of a largest independent set in $G$, and to estimate the Ramsey number $R(3,k)$ is equivalent to estimating the minimum of $\alpha(G)$ over all $n$-vertex triangle-free graphs $G$. 
A local analysis of the hard-core model yields asymptotically the best-known upper bound $R(3,k)\le (1+o(1))k^2/\log k$~\cite{DJPR17b}, though the earlier method of Shearer~\cite{She83} gives better lower-order terms.
The key idea of this local analysis of the hard-core model, often called \emph{local occupancy}, is to consider an appropriate boundary in the graph and invoke spatial independence to show that for any vertex $u\in V$ either $\Pr(u\in X)$ or $\EE\abs{N(u)\cap X}$ is large. 
Given an objective, such as to minimize or maximize $\EE_{G,\blam}\abs{X}$ over graphs $G$ in some class $\cG$, and $\blam$ satisfying some condition, one writes the objective as an optimization problem constrained by facts that follow from the assumption that $X$ is sampled from a hard-core model at fugacity $\blam$ on a graph in the class. 

Our main contribution is to generalize several prior results on the expected size of an independent set from the hard-core model to give lower bounds in terms of degree sequence instead of maximum degree. 
For this, we write $d_u$ for the degree of a vertex $u$ when the graph is clear from context.
We are also able to handle the multivariate setting in some of our bounds. 
Our first result handles arbitrary graphs with a given degree sequence, and is inspired by the Caro--Wei theorem~\cite{Car79,Wei81} and a conjecture of the first author and Kang~\cite{DK25}.

\begin{theorem}\label{thm:multivar-occ-lower-bound}
    Let \(G\) be a graph of maximum degree \(\Delta\). Let \(\blam\in [0,\infty)^{V(G)}\) be such that \(\lambda_u<1/\Delta\) for all \(u\in V(G)\).
    Then
    \[
        \EE_{G,\blam} \abs{I} \ge \sum_{u\in V(G)} \frac{\lam_u}{1+(d_u+1)\lam_u}.
    \]
\end{theorem}

\Cref{thm:multivar-occ-lower-bound} gives the best-known progress towards a conjecture of the first author and Kang~\cite[Conj.~A]{DK25} which states that the bound holds when $\blam=\lam\bone$ for any $\lam>0$. 
In the limit $\lam\to\infty$ the conjecture reduces to the Caro--Wei theorem: $\alpha(G)\ge \sum_{u\in V(G)} \frac{1}{d_u+1}$. Strengthening the conjecture, we believe that the multivariate version should hold for any $\blam$ in the positive orthant.
The bound in \Cref{thm:multivar-occ-lower-bound} is tight by the example of a disjoint union of complete graphs (such that $\blam$ is constant on each component), though we believe that the upper bound on the entries of $\blam$ can be removed.
Integrating the bound, one recovers the lower bound on the multivariate partition function of Lee and the third author~\cite[Thm.~1.1]{LS26}, which in turn generalizes the univariate bound of Sah, Sawhney, Stoner and Zhao~\cite[Thm.~1.7]{SSSZ19}. Those results hold for all (nonnegative real) fugacities, however, so via \Cref{thm:multivar-occ-lower-bound} one only recovers these bounds subject to the stated upper bound on the fugacities.

Further motivation for \Cref{thm:multivar-occ-lower-bound} and our proof technique comes from the local greedy fractional coloring bound of Kelly and Postle~\cite{KP24} and the prospect of generalizing local occupancy to their notion of local demands in fractional coloring. 
It turns out that the upper bound on the entries of $\blam$ prevents us from recovering this result from the proof of \Cref{thm:multivar-occ-lower-bound}. In \Cref{sec:frac} we discuss this and give a generalization of the key lemma that derives fractional coloring bounds from local occupancy that handles the multivariate setting and local demands.

Our second result handles locally sparse graphs with the following natural condition. Suppose that for each vertex $u\in V(G)$ and subgraph $F\subset G[N(u)]$ of the neighborhood of $u$ we have an average degree bound $b$ on the graph $F$. 
We can express this succinctly as saying that the \emph{maximum average degree} $\mad(G[N(u)])$ is bounded by $b$, where 
\[ \mad(H) = \max_{F\subset H, \abs{V(F)}>0} \frac{2\abs{E(F)}}{\abs{V(F)}}. \]
If $b=0$ this assumption is equivalent to assuming that $G$ is triangle-free, and in general the condition implies that each vertex $u$ of $G$ is contained in at most $bd_u/2$ triangles.
Let \(W\colon (0,\infty) \to (0,\infty)\) be the principal branch of the \emph{Lambert \(W\)-function} restricted to \((0,\infty)\), which is the inverse of \(x\mapsto xe^x\) in its range.

\begin{theorem}\label{thm:bdd-loc-mad-occ-lower-bound}
    Let \(G\) be a graph of maximum degree \(\Delta\) with no isolated vertices. Let \(b\in \{0\}\cup[1,\infty)\) and assume \(\mad(G[N(u)])\le b\) for each \(u\in V(G)\). Then there exists \(c=c(b)\in(0,1)\) such that the following holds. Let \(0<\lambda<c/\Delta\) and
    \begin{gather*}
        D_u := d_u (1+\lam)^b \log(1+\lam), \\
        \beta_u := \frac{1+\lam}{\lam} \frac{D_u}{W(D_u) (1+W(D_u))},
        \qquad
        \gam_u := \frac{1+\lam}{\lam} \frac{D_u}{d_u (1+W(D_u))}.
    \end{gather*}
    Then
    \[
        \EE_{G,\lam\bone} \abs{I} \ge \sum_{u\in V(G)} \frac{1}{\beta_u+d_u\gam_u}.
    \]
    In particular, one can choose $c(0)=0.109597$ and for $b\ge 1$, $c(b)=0.0896883/b$.
\end{theorem}

This result is inspired by a lower bound due to Shearer~\cite{She91} on the independence number of a triangle-free graph in terms of its degree sequence, though due to the upper bound on $\lam$ our result does not give comparable bounds on the independence number. 
For the case \(b=0\) of triangle-free graphs, integrating the bound in \Cref{thm:bdd-loc-mad-occ-lower-bound} along a straight line from $\blam=\bzero$ and using standard properties of the Lambert $W$-function~\cite{CGH+96}, we obtain the lower bound
\[ \log Z_G(\blam) \ge \sum_{u\in V(G)}\frac{1}{2d_u}\left[W(d_u\log(1+\lam))^2 + 2W(d_u\log(1+\lam))\right]. \]
This can be compared to a recent result of Buys, van den Heuvel and Kang~\cite{BHK25} who gave a lower bound on $\log Z_G(\blam)$ in terms of the average degree for the univariate case $\blam=\lam\bone$ with $\lam\in[0,1]$.
\Cref{thm:multivar-occ-lower-bound,thm:bdd-loc-mad-occ-lower-bound} give degree-sequence generalizations of previous occupancy fraction bounds~\cite{DdKP21,DJPR17b,DK25}.
Again, we are motivated by fractional coloring and note that a method for removing the upper bound on $\lam$ would likely yield an alternative proof of the local fractional Shearer bound of Martinsson and Steiner~\cite{MS25}. We discuss this further below.

\subsection{Techniques}

A subset of the authors used induction and local occupancy to prove a univariate version of \Cref{thm:multivar-occ-lower-bound} subject to the stronger condition $\lam\le O(1/\Delta^2)$~\cite{DST25}. 
Similarly, they established the case $b=0$ of \Cref{thm:bdd-loc-mad-occ-lower-bound} up to $\lam \le O(1/\Delta^4)$. 
Our analysis significantly simplifies theirs and gives stronger results due to relaxed upper bounds on the fugacity and, in the case of \Cref{thm:multivar-occ-lower-bound}, the generalization to the multivariate case.
One of our technical contributions is to streamline the method, obviating the need for induction and some of the mundane Taylor approximation of the previous work.
While the third author~\cite{SeoNote} showed that the induction technique gives the univariate version of \Cref{thm:multivar-occ-lower-bound} up to $\lam = c/\Delta$ for some absolute constant $c$, the analysis is not straightforward and we do not see how to get $c$ close to $1$.

The cluster expansion is a general technique for handling hard-core models under suitable ``small fugacity'' assumptions. 
The method corresponds to a combinatorial interpretation of the multivariate Taylor expansions (in $\blam$) of functions such as $\log Z_G(\blam)$ and $\Pr_{G,\blam}(u\in X)$. In these expansions, the coefficient terms of order $k$ in $\blam$ correspond to counts of connected subgraphs on at most $k$ vertices in $G$. 
For more details, see the survey of Jenssen~\cite{Jen24}.
Analogous to the work of the first author, Jenssen and Perkins~\cite{DJP21} and a follow-up due to Borbényi and Csikvári~\cite{BC21}, one might expect that the cluster expansion can be used to prove our main theorems. 
For \Cref{thm:multivar-occ-lower-bound}, one could use the fact that a disjoint union of complete graphs maximizes the number of triangles subject to a given degree sequence, and for \Cref{thm:bdd-loc-mad-occ-lower-bound} one could compare terms of order $4$ with the Taylor expansion of the lower bound.
In the univariate case $\blam = \lam\bone$, tight general results on the radius of convergence of the cluster expansion~\cite{She85,SS05} would limit such an analysis to the smaller range
\[ \lam < \frac{(\Delta-1)^{\Delta-1}}{\Delta^\Delta} \sim \frac{1}{e\Delta}. \]
Even then, to get a result up to $\lam = \Omega(\Delta^{-1})$ requires understanding arbitrarily many terms of the expansion, and for triangle-free graphs the fact that given two different graphs one can only see a gap at terms of order at least $4$ means that naive tail bounds limit the analysis to $\lam < O(\Delta^{-2})$.

Our main technical contribution is to generalize local occupancy for the hard-core model to the multivariate setting, i.e.\ where $\blam$ is not necessarily a scalar multiple of $\bone$. For this, we make the following natural definition (cf.~\cite{DK25} where a hierarchy of such conditions of various strengths is discussed).

\begin{definition}
    Given a graph $G$ and $\bbeta,\bgam,\blam\in[0,\infty)^{V(G)}$, we say that $G$ has local $(\bbeta,\bgam,\blam)$-occupancy if, for each vertex $u\in V(G)$ it holds that
\[ \beta_u\Pr_{G,\blam}(X_u) + \gam_u \sum_{v\in N_G(u)}\Pr_{G,\blam}(X_v) \ge 1, \]
where $X_u$ is the indicator random variable for the event that $u$ is occupied.
\end{definition}

There are a number of settings in which tight or asymptotically tight local occupancy parameters can be found that, despite the inherent locality of the method, give rise to best-known or essentially tight bounds on global properties of $G$.

Given a graph $G$ and $\bbeta,\bgam,\blam\in(0,\infty)^{V(G)}$ such that $G$ has local $(\bbeta,\bgam,\blam)$-occupancy, we can write $x_v=\Pr_{G,\blam}(X_v)$ and observe that the definition of local occupancy means that the vector $\bx$ is feasible in the following linear program (LP).
\begin{equation}\label{e:gLO}
\begin{aligned}
    &\min \sum_{u\in V} x_u,\\
    &0 \le x_u &&\forall u\in V\\
    &\beta_u x_u + \gam_u \sum_{v\in N(u)} x_v \ge 1 && \forall u\in V
\end{aligned}
\end{equation}
It is convenient to express this in compact form using matrices. Let $B=\diag(\bbeta)$, $\Gam=\diag(\bgam)$ and let $A$ be the adjacency matrix of $G$. Then the general local occupancy LP is 
\begin{equation}\label{e:gLOm}
\begin{aligned}
    &\min \bone\trans \bx,\\
    &\bzero \le \bx\\
    &(B+\Gam A)\bx \ge \bone.
\end{aligned}
\end{equation}
Note that the objective $\bone\trans \bx$ is exactly $\EE_{G,\blam} \abs{X}$. 
It may be that feasible $\bx$ do not correspond to a hard-core model on a graph, but nonetheless one of the key ideas of local occupancy is to use this LP to give a lower bound on expectations in the hard-core model. 
Versions of this idea for maximization are also of interest. 

The first cases considered were $d$-regular graphs or graphs of maximum degree $\Delta$ in the simple setting where $\bbeta,\bgam,\blam$ are parallel to $\bone$.
In such cases we can exploit symmetry, project the LP to one dimension, and solve for the minimum $x>0$ such that $\beta x + \Delta\gam x \ge 1$ for reals $\beta,\gam>0$, obtaining $x=(\beta+\Delta\gam)^{-1}$. 
This gives rise to the fundamental graph theory problem involving local occupancy: find parameters such that local occupancy holds and $\beta+\Delta\gam$ is minimized.
Our main contribution is to give a novel and general perspective on the LP which allows us to handle bounds involving the degree sequence of $G$ and cases where we do not assume that $\bbeta,\bgam,\blam$ are parallel to $\bone$. 
The analysis brings aspects of spectral graph theory to bear on local occupancy in the hard-core model. 

We defer details of our results on fractional coloring to \Cref{sec:frac}, but note here that the method involves an LP analogous to the dual of the above LP, except that the constraint matrix is not transposed when taking the dual. 
For the required background on linear programming, see e.g.\ the textbook of Schrijver~\cite{Sch03}.

\subsection{Organization}
In \Cref{sec:spectral} we develop a general understanding of local occupancy in the multivariate setting and show how to exploit the main LP~\eqref{e:gLOm} via spectral properties of the constraint matrix derived from the graph and bounds on the entries of $\blam$.
In the two subsequent sections we apply this analysis to prove \Cref{thm:multivar-occ-lower-bound,thm:bdd-loc-mad-occ-lower-bound}.
In \Cref{sec:frac} we turn to local fractional coloring and observe that developing an understanding of this via local occupancy is analogous to solving the dual of LP~\eqref{e:gLOm} (though the constraint matrix involved is the original constraint matrix $B+\Gam A$ rather than the transpose that one would expect from taking the dual).
Finally, in \Cref{sec:limits} we show that to improve substantially upon our results requires more than the LP~\eqref{e:gLOm} by giving an example which shows that the relaxation to the LP is too loose.

\section{A spectral understanding of multivariate local occupancy}\label{sec:spectral}

Suppose that a fixed graph \(G=(V,E)\) on $n$ vertices and \(\bbeta,\bgam,\blam\in(0,\infty)^{V}\) are given such that $G$ has local $(\bbeta,\bgam,\blam)$-occupancy. 
We write $X$ for a random independent set in $G$ distributed according to the hard-core model at fugacity $\blam$. For each \(u\in V\), we use \(\beta_u\) to denote the entry of \(\bbeta\) which corresponds to \(u\), and similarly denote entries of vectors indexed by \(V\).

We may assume that \(G\) is connected and \(n>1\) because $\log Z_G(\blam)$ and $\EE_{G,\blam}\abs{X}$ are additive over the connected components of $G$ and the hard-core model on $K_1$ is trivial.

Throughout, we let \(A\) be the adjacency matrix of \(G\), and let \(D\) be the diagonal matrix indexed by $V\times V$ with $D_{uu}=d_u$ for each \(u\in V(G)\), and as is standard we write \(d_u\) for the degree of \(u\). Let
\[
    B:= \diag(\bbeta),
    \qquad
    \Gam := \diag(\bgam),
    \qquad
    H := B+D\Gam,
\]
and let \(L:=D-A\) be the combinatorial Laplacian of \(G\). Note that by the assumption \(\beta_u,\gam_u>0\) for all \(u\in V(G)\), and that \(G\) has no isolated vertices, the matrices \(B\), \(\Gam\), \(D\), and \(H\) are invertible.

We recall some basic facts from linear algebra. For a real square matrix \(T\), let \(\rho(T)\) be its spectral radius (i.e.\ the maximum absolute value of any eigenvalue of \(T\)). Let \(\opnorm{T}:= \sup_{\bx\neq\bzero} \norm{T\bx}_2 / \norm{\bx}_2\) be the operator norm of \(T\), which satisfies \(\rho(T) \le \opnorm{T}\). Note that if \(T\) is symmetric, \(\rho(T) = \opnorm{T} = \sup_{\bx\neq\bzero} \abs{\bx\trans T \bx} / \norm{\bx}_2^2\).
We also record a proposition on the Neumann series of a matrix. For this result and further basic facts from linear algebra, we refer the reader to a textbook~\cite{HJ85}.

\begin{proposition}\label{prop:Neumann-series}
    Let \(T\) be a real square matrix with \(\rho(T)<1\). Then \(I-T\) is invertible and the inverse can be represented as a Neumann series of $T$ convergent in the operator norm:
    \[
        (I-T)^{-1} = \sum_{k\ge0} T^k.
    \]
\end{proposition}

We also record a proposition giving a multivariate version of known general local occupancy parameters. The proof is identical to the univariate setting, but we give it for completeness. For the univariate proof and details on the spatial Markov property, see the survey of the first author and Kang~\cite{DK25}.

\begin{proposition}\label{prop:genLO}
    Let $G=(V,E)$ be a graph and for $\blam\in(0,\infty)^{V}$ set $\beta_u = 1+1/\lam_u$ and $\bgam=\bone$. Then $G$ has local $(\bbeta,\bgam,\blam)$-occupancy.
\end{proposition}
\begin{proof}
    Fix $u\in V$ and let $Y=\abs{N(u)\cap X}$. Then by the spatial Markov property we have 
    \[ \Pr_{G,\blam}(u\in X) = \frac{\lam_u}{1+\lam_u}\Pr(Y=0) \]
    and by Markov's inequality and the fact that $Y$ takes values in the nonnegative integers we have 
    \[ \EE_{G,\blam}\abs{N(u)\cap X}=\EE Y \ge 1-\Pr(Y=0). \] 
    Note that this is tight if and only if $Y$ is supported on $\{0,1\}$, which holds if and only if $N(u)$ is a clique in $G$. Our choice of parameters now means that 
    \[ \beta_u\Pr_{G,\blam}(u\in X) + \gam_u\sum_{v\in N(u)}\Pr_{G,\blam}(v\in X) \ge \Pr(Y=0) + 1 - \Pr(Y=0) =1.\qedhere \]
\end{proof}

To prove \Cref{thm:multivar-occ-lower-bound,thm:bdd-loc-mad-occ-lower-bound}, our goal is the following. For a graph \(G\) which has local \((\bbeta,\bgam,\blam)\)-occupancy, we recall the LP~\eqref{e:gLOm}:
\begin{equation*}
\begin{aligned}
    &\min \bone\trans \bx,\\
    &\bzero \le \bx\\
    &(B+\Gam A)\bx \ge \bone.
\end{aligned}
\end{equation*}
We prove that the LP has optimum value at least \(\sum_{u \in V(G)} \frac{1}{\beta_u+d_u\gam_u} = \bone\trans H^{-1} \bone\).
We proceed in the natural way, showing that some dual-feasible vector has optimum at least $\bone\trans H^{-1} \bone$ in the dual LP.

The dual LP is
\begin{equation}\label{e:gLOdualm}
\begin{aligned}
    &\max \bone\trans \by,\\
    &\bzero \le \by\\
    &(B+A\Gam)\by \le \bone,
\end{aligned}
\end{equation}
and the goal is equivalent to finding a feasible solution \(\by'\) of \eqref{e:gLOdualm} such that
\begin{equation}\label{eq:y'-entry-sum-lower-bd}
    \bone\trans\by' \ge \bone\trans H^{-1} \bone.
\end{equation}

Before we give our novel method, we point out that the formalism developed thus far allows a neat interpretation of existing proofs. 
If $G$ is $d$-regular and $\bbeta=\beta\bone$, $\bgam=\gam\bone$, $\blam=\lam\bone$ then $\bone$ is an eigenvector of $A$ of eigenvalue $d$, and hence $\bone$ is an eigenvector of the constraint matrix $B+A\Gam$ with eigenvalue $\beta + d\gam$. Then we can take $\by'=\bone/(\beta+d\gam)$ and see that $\by'$ is dual-feasible with objective $\bone\trans H^{-1} \bone$.
The case of a graph $G$ of maximum degree $\Delta$ is slightly more complex as one has to use nonnegativity of the entries of $A$ and $B+A\Gam$, but essentially the same argument shows that 
$\by'=\bone/(\beta+\Delta\gam)$ is dual-feasible.
Together with a suitable analysis of local occupancy parameters in different graph classes~\cite{DJKP20,DdKP21,DKPS20}, this sketch encompasses most applications of local occupancy to the problem of giving lower bounds on the expected size $\EE_{G,\blam}\abs{X}$ of an independent set from the hard-core model. 
A notable exception is the work of Perarnau and Perkins~\cite{PP18} which requires a more general understanding of local occupancy based on larger portions of $G$ than neighborhoods.

To prove \Cref{thm:multivar-occ-lower-bound,thm:bdd-loc-mad-occ-lower-bound}, we show that $B+A\Gam$ is invertible under our assumptions and choose \(\by'=(B+A\Gam)^{-1}\bone\). \Cref{lem:gLOdualm-feasible-solution} states that given the conditions on the local occupancy parameters, this \(\by'\) is well-defined and is a feasible solution of \eqref{e:gLOdualm}. It also provides a bound on the spectral radius of \(L\Gam H^{-1}\), which will be used later to show \eqref{eq:y'-entry-sum-lower-bd}.

\begin{lemma}\label{lem:gLOdualm-feasible-solution}
    Let $G$ be a graph with no isolated vertices. Suppose \(G\) has local $(\bbeta,\bgam,\blam)$-occupancy such that for all $u\in V(G)$,
\begin{equation}\label{eq:cond-gLOdualm-feasible-solution}
    \beta_u,\gam_u,\lam_u>0,
    \qquad
    d_u \frac{\gam_u}{\beta_u} <1,
    \qquad
    \sum_{v\in N(u)} \frac{\gam_v}{\beta_v}\le1.
\end{equation}
Then \(B+A\Gam\) is invertible and \((B+A\Gam)^{-1}\bone\) is a feasible solution of \eqref{e:gLOdualm}. Moreover, \(\rho(L\Gam H^{-1})<1\).
\end{lemma}
\begin{proof}
    We first show that \(B+A\Gam\) is invertible. Since \(B\) is invertible and \(B+A\Gam = B(I+B^{-1}A\Gam)\), it suffices to show that \(I+B^{-1}A\Gam\) is invertible. This follows from \(\rho(B^{-1}A\Gam)<1\), which we show in the following claim.

    \begin{claim}\label{clm:B+AGamma-invertible}
        \(\rho(B^{-1}A\Gam)<1\).
    \end{claim}
    \begin{claimproof}
        Let \(Q:= B^{-1}A\Gam\) and \(T:= (B\Gam)^{1/2}\). Then,
        \[
            S := T Q T^{-1}
            = (B\Gam)^{1/2} Q (B\Gam)^{-1/2}
            = B^{-1/2} \Gam^{1/2} A B^{-1/2} \Gam^{1/2},
        \]
        so \(S\) is symmetric. Also, \(S\) is similar to \(Q\), so
        \[
            \rho(Q) = \rho(S) = \opnorm{S}.
        \]
        Let \(W:= B^{-1/2} \Gam^{1/2}\) so that \(S=WAW\). Since \(L=D-A\) and \(D+A\) are PSD, as the latter is the signless Laplacian, we have \(-D \preceq A\preceq D\), which gives
        \[
            -WDW \preceq S = WAW \preceq WDW.
        \]
        Hence, for every \(\bx\in\RR^{V(G)}\),
        \[
            \abs{\bx\trans S \bx} \le \bx\trans WDW \bx \le \opnorm{WDW} \norm{\bx}_2^2.
        \]
        It follows that, since \(S\) is symmetric and \(WDW\) is diagonal with entries \(d_u\gam_u/\beta_u\),
        \[
            \opnorm{S}
            \le \opnorm{WDW}
            = \max_{u\in V(G)} d_u \frac{\gam_u}{\beta_u}
            < 1.
        \]
        Therefore, \(\rho(Q)<1\) as desired.
    \end{claimproof}
    
    Let \(\by':=(B+A\Gam)^{-1}\bone\) and note that the constraint $(B+A\Gam)\by'\le 1$ is trivially satisfied with equality. Next, we show \(\by'\ge\bzero\), which completes the proof that \(\by'\) is dual-feasible.

    \begin{claim}
        \(\by'\ge\bzero\).
    \end{claim}
    \begin{claimproof}
        \((B+A\Gam)^{-1}\) has a convergent series representation
        \[
            (B+A\Gam)^{-1}
            = (I + B^{-1}A\Gam)^{-1}B^{-1}
            = \sum_{k\ge 0}(-B^{-1}A\Gam)^kB^{-1}
            = \sum_{k\ge 0}(B^{-1}A\Gam)^{2k}(I-B^{-1}A\Gam)B^{-1}.
        \]
        Here, the first series is a standard Neumann series of \((I+B^{-1}A\Gam)^{-1}\) where we are using \(\rho(B^{-1}A\Gam)<1\) given by \Cref{clm:B+AGamma-invertible}. The second series follows from \((-B^{-1}A\Gam)^{2k} + (-B^{-1}A\Gam)^{2k+1} = (B^{-1}A\Gam)^{2k}(I-B^{-1}A\Gam)\) for each \(k\ge0\). From this,
        \[
            \by'
            = (B+A\Gam)^{-1} \bone
            = \sum_{k\ge0} (B^{-1}A\Gam)^{2k} (I-B^{-1}A\Gam) B^{-1} \bone.
        \]
        Since \(B^{-1}A\Gam\) is entrywise nonnegative, to show that \(\by'\ge\bzero\), it suffices to show that \((I-B^{-1}A\Gam) B^{-1} \bone \ge \bzero\). Indeed, the entry of \((I-B^{-1}A\Gam) B^{-1} \bone\) which corresponds to each \(u\in V(G)\) is
        \[
            \frac{1}{\beta_u} \Biggl( 1-\sum_{v\in N(u)} \frac{\gam_v}{\beta_v} \Biggr) \ge0. \qedhere
        \]
    \end{claimproof}

    \noindent
    To complete the proof of \Cref{lem:gLOdualm-feasible-solution}, we show that \(\rho(L\Gam H^{-1})<1\).

    \begin{claim}
        \(\rho(L\Gam H^{-1})<1\).
    \end{claim}
    \begin{claimproof}
        Let \(K:=\Gam H^{-1}\). Since \(K\) is diagonal with positive diagonal, \(K^{1/2}\) is well-defined. Let
        \[
            M
            := K^{1/2} L K^{1/2}
            = K^{1/2} (L\Gam H^{-1}) K^{-1/2}.
        \]
        Then \(\rho(L\Gam H^{-1}) = \rho(M) \le \opnorm{M}\).
        Let \(\mathcal{L} = D^{-1/2} L D^{-1/2}\) be the normalized Laplacian of \(G\), which is PSD and \(\opnorm{\mathcal{L}}\le2\); see e.g.~\cite[Lemma 1.7]{Chu96}.
        For each \(u\in V(G)\), let \(r_u:=d_u\gam_u/(\beta_u+d_u\gam_u)\), and let \(R:=\diag(r_u:u\in V(G)) = D \Gam H^{-1} = DK\). Then
        \(
            M
            = R^{1/2} \mathcal{L} R^{1/2}
        \),
        whence
        \[
            \opnorm{M}
            \le \opnorm{R^{1/2}}^2 \, \opnorm{\mathcal{L}}
            \le 2 \max_{u\in V(G)} r_u.
        \]
        The assumption \(d_u\gam_u/\beta_u<1\) gives \(r_u = 1/(\beta_u/(d_u\gam_u)+1) < 1/2\) for all \(u\in V(G)\), whence \(\opnorm{M} < 1\).
        Therefore, \(\rho(L\Gam H^{-1}) < 1\) as desired.
    \end{claimproof}
    

        

    \noindent
    This completes the proof of \Cref{lem:gLOdualm-feasible-solution}.
\end{proof}

\section{Multivariate setting for general graphs}

In this section, we prove \Cref{thm:multivar-occ-lower-bound} using \Cref{lem:gLOdualm-feasible-solution} and the local occupancy parameters of \Cref{prop:genLO}. We may assume \(\blam\in(0,\infty)^{V(G)}\), since the full case of \(\blam\in[0,\infty)^{V(G)}\) can be covered using a simple limiting argument.
Recall that for $\blam\in(0,\infty)^{V(G)}$ if we take \(\bbeta:=(1 + 1/\lam_u:u\in V(G))\) and \(\bgam:=\bone\), then \(G\) has local \((\bbeta,\bgam,\blam)\)-occupancy. 

\begin{proof}[Proof of \Cref{thm:multivar-occ-lower-bound}]
We start by observing that the assumption $\lam_u < 1/\Delta$ for all $u$ means that the parameters satisfy~\eqref{eq:cond-gLOdualm-feasible-solution}. This follows from the fact that \(\gam_u/\beta_u < \lam_u < 1/\Delta\) for all \(u\in V(G)\).
Applying \Cref{lem:gLOdualm-feasible-solution}, it suffices to show that \(\by':=(B+A\Gam)^{-1} \bone\) of \eqref{e:gLOdualm} satisfies
\[
    \bone\trans \by' \ge \bone\trans H^{-1} \bone.
\]

Using \(B+A\Gam = H-L\Gam\) and \(\rho(L\Gam H^{-1})<1\) from \Cref{lem:gLOdualm-feasible-solution}, we see that the Neumann series for \((I-L\Gam H^{-1})^{-1}\) is well-defined, which gives
\begin{equation}\label{eq:gLOdualm-feasible-solution-Neumann-series}
    \by' = (B+A\Gam)^{-1}\bone = H^{-1}(I-L\Gam H^{-1})^{-1}\bone = H^{-1}\bone + H^{-1}\sum_{k\ge 1}(L\Gam H^{-1})^k \bone.
\end{equation}
So it remains to show that for each \(k\ge1\),
\[
    \bone\trans H^{-1} (L\Gam H^{-1})^k \bone \ge 0.
\]
Since \(\Gam=I\), we have
\[
    \bone\trans H^{-1} (L\Gam H^{-1})^k \bone
    = (H^{-1/2}\bone)\trans \, (H^{-1/2}  L  H^{-1/2})^k \, (H^{-1/2}\bone),
\]
and \(H^{-1/2} L H^{-1/2}\) is PSD as \(L\) is. This completes the proof.
\end{proof}

\section{Univariate setting for bounded local mad graphs}

In this section, we prove \Cref{thm:bdd-loc-mad-occ-lower-bound} on graphs with bounded local mad, again using \Cref{lem:gLOdualm-feasible-solution}.
In the proof of \Cref{thm:multivar-occ-lower-bound}, we show that the series \(\sum_{k\ge1} \bone\trans H^{-1} (L\Gam H^{-1})^k \bone\) is termwise nonnegative, which is not true in general for the local occupancy parameters in \Cref{thm:bdd-loc-mad-occ-lower-bound}. Instead, we show that the first term of the series is positive and dominates the remaining series. Throughout the proof, we frequently use \(0<W(x)<x\) for \(x>0\), which follows from \(x=W(x)e^{W(x)} > W(x)\).

\begin{proof}[Proof of \Cref{thm:bdd-loc-mad-occ-lower-bound}]

Recall that for each \(u\in V(G)\),
\begin{gather*}
    D_u := d_u (1+\lam)^b \log(1+\lam), \\
    \beta_u := \frac{1+\lam}{\lam} \frac{D_u}{W(D_u) (1+W(D_u))},
    \qquad
    \gam_u := \frac{1+\lam}{\lam} \frac{D_u}{d_u (1+W(D_u))}.
\end{gather*}
Let \(\bbeta:=(\beta_u:u\in V)\) and \(\bgam:=(\gam_u:u\in V)\), then \(G\) has local \((\bbeta,\bgam,\lam\bone)\)-occupancy by~\cite[Lemma~21]{DKPS20}. Let \(0<\lambda<c/\Delta\) where \(c\in(0,1)\) and \(c<1/b\) if \(b\ge1\).

First, we check that the parameters satisfy \eqref{eq:cond-gLOdualm-feasible-solution}, for which it suffices to show that \(\gam_u/\beta_u = W(D_u)/d_u\) is less than \(1/\Delta\) for all \(u\in V(G)\).
If \(b=0\), then from \(\log(1+\lambda) < \lambda <1/\Delta\),
\[
    \frac{W(D_u)}{d_u}
    = \frac{D_u}{e^{W(D_u)} d_u}
    = \frac{\log(1+\lam)}{e^{W(D_u)}}
    < \frac{1}{\Delta}.
\]
Suppose now that \(b\ge1\).
Since \(W\) is strictly increasing, \(W(D_u)/d_u < 1/\Delta\) is equivalent to
\[
    D_u < W^{-1}(d_u/\Delta) =  \frac{d_u}{\Delta} e^{d_u/\Delta}.
\]
Substituting \(D_u=d_u(1+\lam)^b\log(1+\lam)\) reduces this to
\[
    (1+\lam)^b \log(1+\lam) < \frac{1}{\Delta} e^{d_u/\Delta},
\]
which holds since \(\lam<1/(b\Delta)\) by the assumption so that
\[
    (1+\lam)^b \log(1+\lam)
    \le e^{b\lam} \lam
    < e^{1/\Delta} \frac{1}{b\Delta}
    \le e^{1/\Delta} \frac{1}{\Delta}
    \le \frac{1}{\Delta} e^{d_u/\Delta}.
\]

Applying \Cref{lem:gLOdualm-feasible-solution}, it suffices to show that \(\by':=(B+A\Gam)^{-1} \bone\) satisfies
\[
    \bone\trans \by' \ge \bone\trans H^{-1} \bone,
\]
where recall \(H=B+D\Gam\). Analogously to \eqref{eq:gLOdualm-feasible-solution-Neumann-series},
\[
    \by' = (B+A\Gam)^{-1}\bone = H^{-1}(I-L\Gam H^{-1})^{-1}\bone = H^{-1}\bone + H^{-1}\sum_{k\ge 1}(L\Gam H^{-1})^k \bone.
\]
So it remains to show that
\begin{equation}
    \sum_{k\ge1} \bone\trans H^{-1} (L\Gam H^{-1})^k \bone \ge 0.
\end{equation}
Let \(S_k:=\bone\trans H^{-1} (L\Gam H^{-1})^k \bone\). As promised in the beginning of the section, we lower bound \(S_1\), upper bound the absolute value of \(\sum_{k\ge2} S_k\), and compare them to obtain the desired inequality \(\sum_{k\ge1} S_k \ge 0\).

\step{Basics} Let \(T:=\Gam H^{-1}\) and \(\bh:=H^{-1}\bone\) so that \(S_k=\bh\trans (LT)^k \bone\). Let \(s:=(1+\lam)^b \log(1+\lam)\). Then \(\lam(1+\lam)^{-1} < \log(1+\lam) < \lam\) gives
\begin{equation}\label{eq:s-lambda-estimate}
    \lam(1+\lam)^{-1} < s < \lam (1+\lam)^b.
\end{equation}
Let \(\btau\) be such that \(T=\diag(\btau)\), and \(w(x):=W(sx)\) so that \(W(D_u)=w(d_u)\). Define the functions \(h\) and \(\tau\) as follows:
\begin{alignat*}{2}
    &h_u
    = \frac{1}{\beta_u+d_u\gam_u}
    = \frac{\lambda}{1+\lambda} \frac{W(D_u)}{D_u}
    = h(d_u),
    &\qquad
    &h(x):=\frac{\lambda}{(1+\lambda)s} \frac{w(x)}{x};
    \\
    &\tau_u
    = \frac{\gam_u}{\beta_u+d_u\gam_u}
    = \frac{W(D_u)}{d_u(1+W(D_u))}
    = \tau(d_u),
    &&\tau(x):=\frac{w(x)}{x(1+w(x))}.
\end{alignat*}

\step{Estimate \(h'\) and \(\tau'\)}
The standard derivative formula for \(W\) gives \(W'(z) = \frac{W(z)}{z(1+W(z))}\), whence
\(
    w'(x)
    = s W'(sx)
    = \frac{w(x)}{x(1+w(x))}
\).
Using this, one can compute that
\[
    h'(x)
    = -\frac{\lam}{(1+\lam)s} \frac{w(x)^2}{x^2(1+w(x))},
    \qquad
    \tau'(x)
    = -\frac{w(x)^2(2+w(x))}{x^2(1+w(x))^3}.
\]
Suppose that $1\le x\le \Delta$. Let \(C:=c (1+c)^b\) so that \(C \ge c > \lam\) and
\[
    0< w(x) = W(sx) < sx\le s\Delta
    < \lam(1+\lam)^b \cdot \Delta
    < c(1+\lam)^b
    \le C,
\]
which in particular gives
\begin{equation}\label{eq:w(x)-bound}
    w(x)/x < s,
    \qquad
    w(x) < C,
    \qquad
    s\Delta < C.
\end{equation}
Also, \(w(x)/x > s e^{-C}\) as
\[
    w(x)
    = W(sx)
    = sx e^{-W(sx)}
    = sx e^{-w(x)}
    > sxe^{-C},
\]
and \(s > \lam(1+\lam)^{-1} > \lam (1+C)^{-1}\) by \eqref{eq:s-lambda-estimate}.
Using these together with \eqref{eq:s-lambda-estimate},
\begin{align*}
    -h'(x)
    &> \frac{\lam}{1+\lam} \frac{s e^{-2C}}{1+C}
    > \frac{\lam^2}{(1+\lam)^2} \frac{e^{-2C}}{1+C}
    > \frac{\lam^2 e^{-2C}}{(1+C)^3}, \\
    -h'(x)
    &< \frac{\lam}{1+\lam} \frac{1}{s} \biggl(\frac{w(x)}{x}\biggr)^2
    < \frac{\lam}{1+\lam} s
    < \lam^2 (1+\lam)^{b-1}, \\
    -\tau'(x)
    &= \biggl(\frac{w(x)}{x}\biggr)^2 \frac{2+w(x)}{(1+w(x))^3}
    > (\lam(1+C)^{-1} e^{-C})^2 \frac{2}{(1+C)^3}
    = \frac{2\lam^2 e^{-2C}}{(1+C)^5}, \\
    -\tau'(x)
    &< 2s^2
    < 2 \lam^2 (1+\lam)^{2b}.
\end{align*}
Thus,
\begin{equation}\label{eq:bdd-loc-mad-h'-tau'-estimate}
    \frac{\lam^2 e^{-2C}}{(1+C)^3}
    < -h'(x)
    < \lam^2 (1+\lam)^{b-1},
    \qquad
    \frac{2\lam^2 e^{-2C}}{(1+C)^5}
    < -\tau'(x)
    < 2 \lam^2 (1+\lam)^{2b}.
\end{equation}
In particular, \(h\) and \(\tau\) are strictly decreasing.

\step{Bound \(\opnorm{T^{1/2}LT^{1/2}}\)}
\eqref{eq:w(x)-bound} gives
\[
    \tau(x)
    = \frac{w(x)}{x(1+w(x))}
    < \frac{w(x)}{x}
    < s,
\]
so \(\tau_u=\tau(d_u)<s\) for all \(u\in V(G)\). Thus,
\begin{equation}\label{eq:LT-T-norm-upper-bound}
    \opnorm{T^{1/2} L T^{1/2}}
    \le \opnorm{T} \opnorm{L}
    \le \max_u \tau_u \cdot 2\Delta
    < s\cdot 2\Delta
    < 2C.
\end{equation}
Here, the second inequality uses \(\opnorm{T}=\max_u \tau_u\) and the standard fact that \(\opnorm{L}\le 2\Delta\) which follows from \(\opnorm{L}\le \opnorm{D-A} \le \opnorm{D}+\opnorm{A}\),
and the last inequality uses \(s\Delta<C\) from \eqref{eq:w(x)-bound}.

\step{Estimate \(S_1\)}
Write
\[
    S_1
    = \bh\trans L \btau
    = \sum_{uv\in E(G)} (h_u-h_v)(\tau_u-\tau_v).
\]
Here and throughout, \(\sum_{uv\in E(G)}\) denotes the sum over unordered edges of \(G\).
Since \(h\) and \(\tau\) are decreasing, \(h_u-h_v\) and \(\tau_u-\tau_v\) have the same sign. By the mean value theorem and \eqref{eq:bdd-loc-mad-h'-tau'-estimate},
\begin{equation}\label{eq:bdd-loc-mad-S1-estimate}
    S_1
    = \sum_{uv\in E(G)} \abs{h_u-h_v} \abs{\tau_u-\tau_v}
    \ge \sum_{uv\in E(G)} \frac{\lam^2 e^{-2C}}{(1+C)^3} \abs{d_u-d_v} \cdot \frac{2\lam^2 e^{-2C}}{(1+C)^5} \abs{d_u-d_v}
    = \frac{2\lam^4 e^{-4C}}{(1+C)^8} \cE_d,
\end{equation}
where \(\cE_d:=\sum_{uv\in E(G)} (d_u-d_v)^2\).

\step{Estimate \(\sum_{k\ge2} S_k\)}
Let
\[
    M := T^{1/2}LT^{1/2}, \qquad
    \bu := T^{-1/2}\bh, \qquad
    \bv := T^{1/2}\bone.
\]
Since \(L\) is PSD, so is \(M\). We claim that for each \(k\ge1\)
\begin{equation}\label{eq:u-M^k-v-ineq}
    \abs{\bu\trans M^k \bv}
    \le \opnorm{M}^{k-1} \bigl( \bu\trans M \bu \, \bv\trans M \bv \bigr)^{1/2}.
\end{equation}
To see this, let \(M=Q \diag(\bmu) Q\trans\) be the spectral decomposition of \(M\) where \(Q\) is orthogonal. Since \(M\) is PSD, \(\bmu=(\mu_v : v\in V(G))\) satisfies \(0\le \mu_v\le \rho(M) = \opnorm{M}\) for each \(v\). Then
\[
    M^k
    = Q \diag(\mu_v^k : v\in V(G)) Q^T
    \preceq Q ( \opnorm{M}^{k-1} \diag(\bmu) ) Q^T
    = \opnorm{M}^{k-1} M.
\]
The Cauchy--Schwarz inequality for the PSD form induced by \(M^k\) gives
\[
    \abs{\bu\trans M^k \bv}^2
    \le (\bu\trans M^k \bu) (\bv\trans M^k \bv),
\]
and applying \(M^k \preceq \opnorm{M}^{k-1} M\) to each factor derives \eqref{eq:u-M^k-v-ineq}.

Now observe \(M=T^{1/2}(LT)T^{-1/2}\), and hence \(M^k=T^{1/2}(LT)^kT^{-1/2}\). Therefore
\[
    S_k
    = \bh\trans (LT)^k \bone
    = \bu\trans M^k \bv.
\]
Moreover,
\[
    \bu\trans M \bu = \bh\trans L \bh,
    \qquad
    \bv\trans M \bv = \bone\trans TLT\bone = \btau\trans L \btau,
\]
where we used \(T\bone=\btau\). Thus, by \eqref{eq:u-M^k-v-ineq},
\[
    \abs{S_k}
    \le \opnorm{M}^{k-1} ( \bh\trans L \bh \, \btau\trans L \btau )^{1/2}.
\]
By the mean value theorem and \eqref{eq:bdd-loc-mad-h'-tau'-estimate},
\[
    \bh\trans L \bh
    = \sum_{uv\in E(G)} (h_u-h_v)^2
    \le \sum_{uv\in E(G)} ( \lam^2(1+\lam)^{b-1} )^2 (d_u-d_v)^2
    = (\lam^2 (1+\lam)^{b-1})^2 \cE_d,
\]
and similarly
\[
    \btau\trans L \btau
    \le (2\lam^2(1+\lam)^{2b})^2 \cE_d.
\]
Combining the above inequalities together with \(\opnorm{M}<2C\) from \eqref{eq:LT-T-norm-upper-bound},
\begin{align*}
    \abs{S_k}
    &\le \opnorm{M}^{k-1} ( \bh\trans L \bh \, \btau\trans L \btau )^{1/2}
    \\&\le (2C)^{k-1} \cdot \lam^2(1+\lam)^{b-1} \cdot 2\lam^2(1+\lam)^{2b} \cdot \cE_d
    \\&= (2C)^{k-1} 2\lam^4 (1+\lam)^{3b-1} \cE_d
    \\&\le (2C)^{k-1} 2\lam^4 (1+c)^{(3b-1)_+} \cE_d,
\end{align*}
where \((3b-1)_+ := \max\{3b-1,0\}\).
Thus, assuming \(C<1/2\) (which will follow from the eventual choice of \(c\)),
\begin{equation}\label{eq:bdd-loc-mad-remaining-series-estimate}
    \abs[\bigg]{\sum_{k\ge2} S_k}
    \le 2\lam^4 (1+c)^{(3b-1)_+} \frac{2C}{1-2C} \cE_d.
\end{equation}

\step{Estimate \(\sum_{k\ge1} S_k\)}
Combining \eqref{eq:bdd-loc-mad-S1-estimate} and \eqref{eq:bdd-loc-mad-remaining-series-estimate},
\[
    \sum_{k\ge1} S_k
    \ge 2\lam^4 \biggl( \frac{e^{-4C}}{(1+C)^8}
        - (1+c)^{(3b-1)_+} \frac{2C}{1-2C} \biggr) \cE_d.
\]
Thus, it suffices to have
\[
    f(c) := \frac{e^{-4C}}{(1+C)^8} \ge (1+c)^{(3b-1)_+} \frac{2C}{1-2C} =: g(c).
\]
Since \(C=c(1+c)^b\) is strictly increasing in \(c\), in the range of \(c\) where \(C<1/2\), \(f(c)\) is strictly decreasing while \(g(c)\) is strictly increasing. In addition, both \(f\) and \(g\) are continuous, \(f(0)=1>0=g(0)\), and \(g(c)\to\infty\) as \(C\to1/2\). Thus, there is a unique admissible \(c\) such that \(f(c)=g(c)\).

\medskip

Finally, we estimate \(c(b)\).

\step{Case \(b=0\)} The condition becomes
\[
    \frac{e^{-4C}}{(1+C)^8} = \frac{2C}{1-2C},
\]
whose solution is \(C\approx 0.1095972\). Since \(C=c(1+c)^b = c\), the optimal \(c\) is \(\approx 0.1095972\).

\step{Case \(b\ge1\)} The condition becomes
\[
    \frac{e^{-4C}}{(1+C)^8} = (1+c)^{3b-1} \frac{2C}{1-2C}.
\]
Set \(c=\eta/b\) where \(\eta>0\) will be chosen later. Then
\[
    bc \le \eta,
    \qquad
    (1+c)^b \le e^{bc} \le e^\eta,
    \qquad
    C = c(1+c)^b \le \eta e^\eta.
\]
Clearly, the right-hand side is increasing on \(0<C<1/2\).
The left-hand side is decreasing on \(C>0\), since for \(x>0\),
\[
    \frac{d}{dx} \frac{e^{-4x}}{(1+x)^8}
    = -\frac{4(x+3)e^{-4x}}{(1+x)^9} < 0.
\]
Thus, it suffices to have
\[
    \frac{e^{-4u}}{(1+u)^8} \ge e^{3\eta} \frac{2u}{1-2u},
    \qquad
    u:=\eta e^\eta,
\]
and solving this gives \(\eta\approx 0.08968838\). Note that if \(\eta<0.08968838\), then \(C\le \eta e^\eta < 1/2\).
\end{proof}


\section{Fractional coloring}\label{sec:frac}

The main result of this section gives a generalization of the main fractional coloring result derived via local occupancy, the case $r=1$ of~\cite[Lemma 3]{DJKP20}. 
We also give the generalization to fractional coloring with local demands in the sense of Kelly and Postle~\cite{KP24}.
The idea traces back to Molloy and Reed~\cite[\S21.3]{MR02} who studied a uniform random maximum independent set (the limit as a univariate fugacity tends to infinity). The first author, de Joannis de Verclos, Kang and Pirot~\cite{DJKP20} gave the generalization to arbitrary finite univariate fugacity, and they noted a certain locality property that is distinct from the local demands definition of Kelly and Postle. 
We define these concepts below and connect them to local occupancy.

Given a graph $G=(V,E)$, let $\ind(G)\subset \RR^V$ be the \emph{independence polytope} given by the convex hull of the indicator vectors of independent sets in $G$. That is, for each independent set $I\in\cI(G)$ we let $\bx^I$ be given by $\bx^I_v=1$ for $v\in I$ and $\bx^I_v=0$ otherwise, and set $\ind(G)$ to be the convex hull of the $\bx^I$. 
Clearly, points in $\ind(G)$ correspond to probability distributions on the independent sets of $G$. 
A fractional $k$-coloring of $G$ is a vector $\bq\in\ind(G)$ such that for each vertex $v\in V$ we have $q_v\ge 1/k$. 
Kelly and Postle~\cite{KP24} gave a natural local version of this and defined, for a function $f:V\to [0,1]\cap\QQ$, a \emph{fractional coloring with demand $f$} to be a vector $\bq\in \ind(G)$ such that for each $v$ we have $q_v\ge f(v)$.
As an example, they gave several proofs of the ``local fractional greedy bound'' which states that for all graphs $G$, the vector $\bq$ with $q_v = \frac{1}{d_v+1}$ is in $\ind(G)$.
Via local occupancy, we describe a simple linear condition of dimension $\abs{V}$ that suffices for a vector to be in $\ind(G)$.

\begin{theorem}\label{thm:LOdemands}
    Let $G=(V,E)$ be a graph and suppose that there exist $\bbeta,\bgam,\blam\in(0,\infty)^V$ such that all induced subgraphs $H\subset G$ have local $(\bbeta,\bgam,\blam)$-occupancy. 
    Let $B=\diag(\bbeta)$, $\Gam=\diag(\bgam)$, and let $A$ be the adjacency matrix of $G$.
    Then for any vector $\bq\ge \bzero$ such that $(B+\Gam A)\bq \le \bone$ we have $\bq\in\ind(G)$.
\end{theorem}

Given local occupancy, this result reduces the problem of finding a fractional coloring with demand $f$ to showing that the vector $\bq$ with $q_v=f(v)$ satisfies the constraint $(B+\Gam A)\bq\le \bone$. 
This is straightforward when $\bone$ is an eigenvector of $B+\Gam A$, which holds when $G$ is $d$-regular, and when $G$ has maximum degree $\Delta$ we choose parameters such that $(B+\Gam A)\bone \le \max_{v\in V}(\beta_v+\Delta\gam_v) \bone$.
Prior results~\cite{DJKP20,DdKP21,DKPS20} essentially proceed in this manner. 
Analogous to the proofs in preceding sections, subject to a condition on $\blam$ that makes $B+\Gam A$ invertible, one can compute $(B+\Gam A)^{-1}\bone$ to obtain fractional colorings from local occupancy. We omit the details, however, as the restriction on $\blam$ we make to push this through seems to preclude interesting consequences for the independence polytope of $G$.

\begin{proof}[Proof of \Cref{thm:LOdemands}]
    It is convenient to write $\mu$ for the hard-core model on $G$ at fugacity $\blam$, and for an (induced) subgraph  $H\subset G$ we write $\mu_{H}$ for the hard-core model on $H$ with fugacity obtained from $\blam$ by deleting entries corresponding to vertices not in $H$. We also extend the notation so that 
    \[ \mu_{H}(v) = \sum_{I\in \cI(H) : v\in I} \mu_{H}(I) \]
    is the marginal of a vertex $v$.
    
    For continuous time $t\ge 0$, consider a weight vector $\bw_t\in[0,\infty)^{\cI(G)}$ defined as follows. Initialize $\bw_0=0$ and write $w^I_t$ for the component of $\bw_t$ corresponding to the independent set $I\in\cI(G)$. 
    For each vertex $v\in V$ let $\bw_t(v) = \sum_{I\ni v}w^I_t$ be the weight of independent sets containing $v$, and define $\bw_t(G) = \sum_{I\in\cI(G)}w_t^I$ to be the total weight. 
    Write $U_t$ for the ``active set'' of vertices $v$ such that $\bw_t(v)< \bq_v$. Then $U_0 = V$.
    
    It is convenient to suppose that $\bw_t$ evolves with time according to the differential equation
    \[ \frac{d\bw_t^I}{dt} = \mu_{G[U_t]}(I). \]
    The right-hand side is not continuous, however, and we prove that there is a solution by constructing it manually following discrete time steps\footnote{This is analogous to previous discrete-time versions of the proof~\cite{MR02,DJKP20}, though we find the continuous-time intuition makes the proof more approachable.}.
    Intuitively, while a vertex $v$ is in $U_t$ its weight is monotone increasing with rate bounded away from zero since 
    \[ \frac{d \bw_t(v)}{dt} = \sum_{I\ni v}\frac{dw^I_t}{dt} = \sum_{I\ni v} \mu_{G[U_t]}(I) = \mu_{G[U_t]}(v) \ge \frac{\lam_v}{Z_G(\blam)}. \]
    Thus, we can construct each $w_t^I$ as a piecewise-linear function with finitely many discontinuities. 
    
    Fix $I\in\cI(G)$. Starting from \(w_0^I=0\) and \(t=0\), suppose that \(w_{t_j}\) and \(U_{t_j}\) have been defined such that $U_{t_j}\ne\emptyset$. On a time interval beginning at \(t_j\), define
    \[
        w_t^I=w_{t_j}^I+(t-t_j)\mu_{G[U_{t_j}]}(I).
    \]
    For \(v\in U_{t_j}\), as justified above, the rate of increase of \(w_t(v)\) is strictly positive, while vertices outside \(U_{t_j}\) have rate \(0\). Hence there is a first time \(t_{j+1}>t_j\) at which some vertex of \(U_{t_j}\) reaches its target value \(q_v\).
    At time \(t_{j+1}\) the cardinality of the active set decreases by at least one. Since \(V(G)\) is finite, this construction takes finitely many steps until we reach $U_t=\emptyset$ and all weights are subsequently constant. This process therefore defines a piecewise-linear absolutely continuous function \(w_t\), satisfying for almost every $t$
    \[
        \frac{d w_t^I}{dt}=\mu_{G[U_t]}(I).
    \]
    For each vertex $v$, let $T_v$ be the first time at which $\bw_t(v)=q_v$, and let $T$ be the first time at which $U_t=\emptyset$. By the above arguments, these times exist and are finite.
    
    Note that by linearity and the fact that probabilities sum to $1$ we have for all $t<T$
    \[ \frac{d\bw_t(G)}{dt} = 1.\]
    Then $\bw_1(G)\le 1$ and hence there is a probability distribution $\nu$ on $\cI(G)$ with $\nu(I)\ge w_1^I$. 
    To complete the proof is equivalent to showing that the marginals $\nu(v)$ are at least $q_v$, which is implied by the fact $T_u\le 1$ for all $u$ that we establish below.

    For all $t\in[0,T)$, the local occupancy property in the induced subgraph $G[U_t]$ of $G$ gives for all $u\in U_t$ that
    \[ \beta_u\frac{d \bw_t(u)}{dt} + \gam_u\sum_{v\in N(u)\cap U_t}\frac{d \bw_t(v)}{dt} \ge 1. \]
    Fixing $u\in V$ and integrating the above inequality over $t\in[0,T_u)$ we have 
    \begin{equation}\label{eq:Tubound}
        \beta_u q_u + \gam_u\sum_{v\in N(u)} q_v \ge \beta_uq_u + \gam_u\sum_{v\in N(u)}\bw_{\min\{T_u,T_v\}}(v) \ge T_u.
    \end{equation}
    Hence it suffices to have for each $u$ that
    \[ \beta_uq_u + \gam_u\sum_{v\in N(u)}q_v \le 1. \]
    This is the condition $(B+\Gam A)\bq \le \bone$ from the statement of the theorem.
\end{proof}

There may be substantial slack in the final inequality in~\eqref{eq:Tubound} in cases where neighbors $v$ of $u$ have $T_v < T_u$. 
It appears hard to avoid this loss with these methods, and this lends credence to the idea that local occupancy performs better when we have homogeneous behavior, e.g.\ of marginals. 

Note that a previous result deriving fractional colorings from local occupancy~\cite[Lem.~3]{DJKP20} handles a case of uniform demands $\bq=q \bone$ where $q=1/\max_{u\in V(G)}\{\beta_u + \gam_u d_u\}$, but gives an additional guarantee that is equivalent (in our setup) to the statement that $T_u \le (\beta_u + d_u\gam_u)q\le 1$. Since our proof is a generalization of those techniques, one can easily guarantee that property here---see~\eqref{eq:Tubound}. We prefer to emphasize the setting of local demands, however.

\section{Limitations of local occupancy}\label{sec:limits}

One cannot rely on the local occupancy LP~\eqref{e:gLOm} alone for arbitrarily large fugacities.
The smallest graph which is not a union of complete graphs is $G=K_{1,2}$, and this already provides a counterexample. For $\blam=(7/5,7/5,7/5)$ and the local occupancy parameters from \Cref{prop:genLO}, it is easy to check that when the vertex of degree $2$ is first in the indexing order, $\bx=(1,0,0)\trans$ is feasible in the LP built from $K_{1,2}$ (in fact, it is optimal). At this $\bx$ the objective value is $1$, but the function we hope to show is a lower bound on the expected size of an independent set evaluates to $497/494 > 1$. 
Overcoming this loss in the LP relaxation in the case of heterogeneous degrees and large fugacities is a problem for future work.

\bibliographystyle{habbrv}
\bibliography{bib}

\end{document}